\documentclass[11pt]{article}

\usepackage[a4paper,margin=2.5cm]{geometry}
\usepackage[utf8]{inputenc}
\usepackage[T1]{fontenc}
\usepackage{lmodern}
\usepackage{amsmath,amssymb,amsthm,mathtools}
\usepackage{enumitem}
\usepackage{xcolor}
\usepackage{hyperref}

\hypersetup{colorlinks=true,linkcolor=blue,citecolor=blue,urlcolor=blue}

\newcommand{\F}{\mathbb F}

\newcommand{\1}{\mathbf 1}
\newcommand{\eps}{\varepsilon}
\newcommand{\wh}{\widehat}
\newcommand{\ol}{\overline}

\newcommand{\norm}[1]{\left\|#1\right\|}

\newcommand{\Prob}{\mathcal P}
\newcommand{\sgn}{\operatorname{sgn}}
\newcommand{\Tr}{\operatorname{Tr}}

\newtheorem{theorem}{Theorem}[section]
\newtheorem{proposition}[theorem]{Proposition}
\newtheorem{lemma}[theorem]{Lemma}
\newtheorem{corollary}[theorem]{Corollary}

\newtheorem{conjecture}[theorem]{Conjecture}
\newtheorem{claim}[theorem]{Claim}
\newtheorem{remark}[theorem]{Remark}

\title{Erd\H{o}s--Falconer distance conjecture from an analytic perspective}

\author{Le Quang Ham \and Dung The Tran}

\date{}

\begin{document}
\maketitle

\begin{abstract}
Let \(q\) be an odd prime power and let \(V=\F_q^{2m}\), equipped with
\(Q(x)=x_1^2+\cdots+x_{2m}^2\). We develop a semidefinite Delsarte
framework for the two-set Erd\H{o}s--Falconer distance problem over \(V\).
The framework reduces the natural \(q^m\)-scale positive-proportion
theorem to a uniform \(L^1\) anti-concentration statement for positive
convex combinations of classical Kloosterman sums. Assuming this
Kloosterman anti-concentration conjecture, we prove that, for every
\(0<\alpha<\frac12\), there exists a constant \(C_{m,\alpha}\) such that
\[
    \min\{|E|,|F|\}\ge C_{m,\alpha}q^m
    \quad\Longrightarrow\quad
    |\Delta_Q^\times(E,F)|>\alpha(q-1)
\]
for all \(E,F\subset\F_q^{2m}\).
More generally, a \(q^{-\theta}\)-level version of the Kloosterman
input yields the geometric threshold \(q^{m+\theta}\).
Unconditionally, a second-moment estimate gives, for every fixed
\(0<\alpha<1\),
\[
    \min\{|E|,|F|\}\ge C_{m,\alpha}q^{m+\frac12}
    \quad\Longrightarrow\quad
    |\Delta_Q^\times(E,F)|>\alpha(q-1).
\]
For the smaller range \(0<\alpha<\frac12\), a cubic-moment argument
improves the \(L^1\) lower bound from the \(q^{-\frac12}\) scale to the
\(q^{-\frac13}\) scale and consequently improves the distance
threshold to \(q^{m+\frac13}\).
The proof uses positive semidefinite \(2\times2\) Gram matrices on
quadratic frequency shells, the shell Fourier transform in even
dimension, and a minimax separation argument that produces a uniform
signed combination of Kloosterman columns. We also provide evidence
for the constant-scale Kloosterman conjecture and discuss its
limitations near full support.
\end{abstract}

\section{Introduction}
Throughout, $q$ is an odd prime power, $m\ge1$, and
\[
    V=\F_q^{2m},
    \qquad
    Q(x)=x_1^2+\cdots+x_{2m}^2.
\]
For $E\subset V$, define 
\[
    \Delta_Q(E):=\{Q(x-y):x, y\in E\}, ~~~\Delta_Q^\times(E):=\Delta_Q(E) \cap \F_q^\times.
\]
For two sets \(E,F\subset V\), we also write
\[
    \Delta_Q(E,F):=\{Q(x-y):x\in E,\ y\in F\}, ~~~\Delta_Q^\times(E,F):=\Delta_Q(E,F) \cap \F_q^\times.
\]
The well--known Erd\H{o}s--Falconer distance conjecture states that if $|E|\gg q^m$, then $|\Delta_Q(E)|\gg q$.

In general dimension \(d\), Iosevich and Rudnev
\cite{IosevichRudnev2007} proved that if \(|E|\gtrsim q^{\frac{d+1}{2}}\), then \(|\Delta_Q(E)|\gtrsim q\).
For \(d=2m\), this gives the exponent \(m+\frac12\). The
Iosevich--Rudnev exponent is sharp in odd dimensions, as shown by a
construction in \cite{HIKR11}. Although a series of papers has studied this problem (see, for example, \cite{chapman, cheong, JF, koh2023, mathz, KohPhamVinh2021, KSun, BL, mu2, pham11, pham2020, pham, PX26,  Pham-Yoo, ZhangLP}), the conjecture in even dimensions has remained open for more than twenty years. 
In the plane, the current record exponents are \(\frac{5}{4}\) over prime fields \cite{mu2} and \(\frac{4}{3}\) over arbitrary finite fields \cite{chapman}. In higher even dimensions, Pham and Xue, in
a breakthrough paper \cite{PX26}, introduced a novel method to break the \(L^2\) spherical Stein--Tomas exponent in \(\mathbb F_p^4\) for an odd prime \(p\); as an application of their main result, they obtained the first improvement of the exponent from \(m+\frac{1}{2}=2+\frac{1}{2}\) to \(\frac{77}{31}\).

Interestingly, a recent linear programming (LP)-based argument of Zhang \cite{ZhangLP} establishes a stronger exponent \(m+\frac{1}{3}\), $m\ge 2$, for this problem over arbitrary finite fields
\(\mathbb F_q\). More precisely, it shows that if
\[
    |\Delta_Q(E)\cap \mathbb F_q^\times|\le \alpha(q-1)
\]
for some \(0<\alpha<\frac{1}{2}\), then
\[
    |E|\ll_\alpha q^{m+\frac{1}{3}}.
\]
We briefly compare the two approaches: restriction theory and the
linear programming method. At the level of one-set support exponents,
the LP result is currently sharper: \(m+\frac{1}{3}\) improves on
\(\frac{77}{31}\), the exponent obtained via restriction theory in
\(\mathbb F_p^4\), and it holds over arbitrary finite fields
\(\mathbb F_q\), rather than only over prime fields. This sharper
exponent, however, comes with a weaker conclusion. The LP
theorem is a one-set support theorem: it shows that if \(E\) determines
only an \(\alpha\)-proportion, \(\alpha<\frac{1}{2}\), of the nonzero distances, then \(E\) cannot be too large. 
This is qualitatively different from an \(L^2\) estimate for distance multiplicities, which restriction theory targets directly and which can yield almost-all conclusions of the form
\((1-o(1))q\) in the settings where sufficiently strong restriction
estimates are available.

The distinction relevant here is structural. Zhang's scalar LP
argument uses the positivity of \(|\widehat{\mathbf 1_E}|^2\) on dual
quadratic levels. For two sets, the corresponding mixed Fourier term
\[
    \widehat{\mathbf 1_E}\,\overline{\widehat{\mathbf 1_F}}
\]
has no fixed sign, so the scalar certificate does not directly extend
to cross-distances. Restriction and incidence methods can instead
access two-set and pinned questions through estimates for distance
multiplicities. This motivates a matrix-valued positivity principle
adapted to cross-correlations. The present paper develops such a
principle by replacing scalar Delsarte positivity with positive
semidefinite \(2\times2\) Gram matrices on frequency shells. It leads
to a reduction of the two-set support problem to a one-dimensional
anti-concentration problem for Kloosterman sums.

Throughout the paper, we write
\[
    K(a):=\sum_{r\in\mathbb F_q^\times}
    \chi\left(r+\frac{a}{r}\right),
    \qquad a\in\mathbb F_q,
\]
for the classical Kloosterman sum over \(\mathbb F_q\). For a nonempty
set \(T\subseteq\mathbb F_q^\times\), write
\begin{align}\label{definition-simplex}
    \Prob(T):=\left\{\lambda=(\lambda_t)_{t\in T}:
    \lambda_t\ge0,\ \sum_{t\in T}\lambda_t=1\right\}.
\end{align}
The analytic input needed by the semidefinite argument is the following
uniform \(L^1\) anti-concentration statement for positive convex
combinations of Kloosterman columns.

\begin{conjecture}
\label{conj:kloosterman-L1}
For every \(0<\alpha<\frac12\), there exists a constant \(c_\alpha>0\)
such that the following holds for every odd prime power \(q\): if
\(\varnothing\ne T\subseteq\mathbb F_q^\times\) satisfies
\[
    |T|\le \alpha(q-1),
\]
then every probability measure \(\lambda\in\Prob(T)\) satisfies
\[
    \Phi_T(\lambda)
    :=
    \frac1{q-1}\sum_{s\in\mathbb F_q^\times}
    \left|\sum_{t\in T}\lambda_tK(st)\right|
    \ge c_\alpha.
\]
\end{conjecture}

The presently known evidence and substitutes for
Conjecture~\ref{conj:kloosterman-L1} may be summarized as follows.
Section~\ref{sec4} proves a constant-scale lower bound for certain
concentrated measures and for the uniform measures on some structured
supports formed from multiplicative cosets.  It also shows that a
uniform positive lower bound fails for full support and for a sequence
of supports whose densities tend to \(1\).  For arbitrary probability
measures, we prove an unconditional \(q^{-\frac12}\)-level bound for
every fixed support density \(0<\alpha<1\), and a stronger
\(q^{-\frac13}\)-level bound when \(0<\alpha<\frac12\).  The examples
near full support do not rule out a constant-scale estimate for any
fixed \(\alpha\in(\frac12,1)\).

Conjecture~\ref{conj:kloosterman-L1} is deliberately stated without
reference to the distance problem.  It asserts that no positive convex
combination of the Kloosterman columns indexed by a set of density
strictly less than \(\frac{1}{2}\) can have small average absolute value in the \(s\)-variable. 
A notable feature of the reduction is that this
analytic conjecture is dimension-free: once the Kloosterman \(L^1\)
lower bound is known over \(\mathbb F_q^\times\), the same input gives
the \(q^m\)-scale conclusion in every even dimension \(2m\).

\begin{theorem}
\label{thm:kloosterman-L1-distance}
Assume Conjecture~\ref{conj:kloosterman-L1} holds for a fixed
\(0<\alpha<\frac12\). Then, there exists a constant \(C_{m,\alpha}\)
such that, for all \(E,F\subset V\),
\[
    \min\{|E|,|F| \} \ge C_{m,\alpha}q^m
\]
implies
\[
    |\Delta_Q^\times(E,F)|>\alpha(q-1).
\]
Equivalently, it is enough to prove that:
if \(T\subseteq\mathbb F_q^\times\) satisfies
\(|T|\le\alpha(q-1)\), and if \(E,F\subset V\) satisfy
\[
    |E|=|F|=L,
    \qquad
    \Delta_Q^\times(E,F)\subseteq T,
\]
then
\[
    L\le C_{m,\alpha}q^m.
\]
\end{theorem}

The same proof gives a quantitative dictionary between Kloosterman
anti-concentration and geometric distance thresholds.

\begin{proposition}
\label{prop:quantitative-transfer}
Fix \(0<\alpha<1\) and \(0\le\theta<1\). Suppose that there is a
constant \(c_{\alpha,\theta}>0\) such that, for every odd prime power
\(q\), every nonempty \(T\subseteq\mathbb F_q^\times\) with
\(|T|\le\alpha(q-1)\), and every \(\lambda\in\Prob(T)\), one has
\[
    \frac1{q-1}\sum_{s\in\mathbb F_q^\times}
    \left|\sum_{t\in T}\lambda_tK(st)\right|
    \ge c_{\alpha,\theta}q^{-\theta}.
\]
Then, there exists \(C_{m,\alpha,\theta}\) such that
\[
    \min\{|E|, |F|\} \ge C_{m,\alpha,\theta}q^{m+\theta}
    \quad\Longrightarrow\quad
    |\Delta_Q^\times(E,F)|>\alpha(q-1).
\]
\end{proposition}

Proposition~\ref{prop:quantitative-transfer} makes the hierarchy of
consequences explicit.  The conjectural constant-scale estimate
(\(\theta=0\)) gives the natural \(q^m\) threshold for
\(0<\alpha<\frac12\).  Unconditionally, the second-moment argument of
Subsection~\ref{subsec:discussion-theta-half} gives
\(\theta=\frac12\) for every fixed \(0<\alpha<1\), while the cubic
argument of Subsection~\ref{subsec:discussion-theta-third} improves
this to \(\theta=\frac13\) only in the range
\(0<\alpha<\frac12\).  Thus the unconditional conclusions are the
following.

\begin{corollary}
\label{cor:unconditional-baseline}
Let \(m\ge1\).
\begin{enumerate}[label=\textup{(\roman*)}]
\item For every fixed \(0<\alpha<1\), there exists
      \(C_{m,\alpha}>0\) such that, for all
      \(E,F\subset V\),
      \[
          \min\{|E|,|F|\}\ge C_{m,\alpha}q^{m+\frac12}
          \quad\Longrightarrow\quad
          |\Delta_Q^\times(E,F)|>\alpha(q-1).
      \]
\item For every fixed \(0<\alpha<\frac12\), there exists
      \(C_{m,\alpha}>0\) such that, for all
      \(E,F\subset V\),
      \[
          \min\{|E|,|F|\}\ge C_{m,\alpha}q^{m+\frac13}
          \quad\Longrightarrow\quad
          |\Delta_Q^\times(E,F)|>\alpha(q-1).
      \]
\end{enumerate}
\end{corollary}

The appearance of Kloosterman sums is not itself new: it follows from
the classical Fourier evaluation of quadratic shells
\cite{IosevichRudnev2007}. The corresponding shell eigenvalue matrix
also appears in Zhang's scalar one-set Delsarte formulation
\cite{ZhangLP}. The new mechanism here is a semidefinite two-set
Delsarte reduction. In the one-set problem, the
support condition on \(E-E\) controls an inner distribution and scalar
Delsarte positivity is available.  In the two-set problem, the
assumption \(\Delta_Q^\times(E,F)\subseteq T\) controls only the
cross-distances, while the two self-distance distributions remain
unrestricted.  We encode the two self-correlations and their mixed
correlation in a \(2\times2\) positive semidefinite Gram matrix on each
frequency shell.  Testing these matrices against \(2\times2\)
positive semidefinite multipliers and applying minimax separation
produces the \(L^1\) Kloosterman input used in
Conjecture~\ref{conj:kloosterman-L1}.

Existing restriction and incidence methods address two-set or pinned
distance questions through estimates for distance multiplicities; see,
for example, \cite{chapman,KSun,KohPhamVinh2021,mu2}. The result here
is instead a support theorem for two sets.  Accordingly, the claimed
novelty is neither a new evaluation of the quadratic-shell Fourier
transform nor a new occurrence of Kloosterman sums, but the
\(2\times2\) semidefinite positivity mechanism and the resulting
dimension-free reduction of the \(q^m\)-scale two-set problem to a
Kloosterman \(L^1\) anti-concentration statement.
The auxiliary shell identities and one-sided cubic inequalities used
in Section~\ref{sec4} come from Zhang's argument; the present framework
combines, for \(0<\alpha<\frac12\), the resulting
\(q^{-\frac13}\)-level input with the two-set transfer principle.

The even-dimensional hypothesis is essential for this particular
reduction.  When \(d=2m\), the Gauss-sum factor in the Fourier transform
of a quadratic shell has no surviving quadratic character, and the
resulting shell eigenvalues are classical Kloosterman sums.  In odd
dimension an extra quadratic character remains, leading instead to
Sali\'e-type sums.  This analytic distinction is consistent with the
known odd-dimensional examples showing that the natural even-dimensional
\(q^{d/2}\) threshold cannot hold in general.

The contribution therefore has two distinct layers.  The conceptual
and conditional part is the semidefinite framework together with the
dimension-free reduction: Conjecture~\ref{conj:kloosterman-L1} would
give the natural \(q^m\) threshold.  The unconditional part consists
of the moment estimates: the second moment recovers the
\(q^{m+\frac12}\) two-set threshold for every fixed
\(0<\alpha<1\), and the cubic moment improves it to
\(q^{m+\frac13}\) when \(0<\alpha<\frac12\).  The model cases and
near-full-support examples in Section~\ref{sec4} provide evidence and
limitations for the analytic conjecture, but are logically separate
from these two unconditional thresholds.

\paragraph{Main idea and sketch of proof.}
Assume that \(E,F\subset\mathbb F_q^{2m}\) are large but that all their
nonzero cross-distances lie in a small set
\(T\subset\mathbb F_q^\times\). For each nonzero frequency radius \(s\),
we restrict \(\widehat{\mathbf 1_E}\) and
\(\widehat{\mathbf 1_F}\) to the shell
\(\Sigma_s=\{\xi:Q(\xi)=4s\}\) and form the corresponding \(2\times2\)
Gram matrix. These matrices are positive semidefinite. The shell
Fourier formula rewrites their entries in physical space; in even
dimension the off-diagonal entry is controlled by the Kloosterman
coefficients \(K(st)\), \(t\in T\).

Conjecture~\ref{conj:kloosterman-L1} is then converted, via von
Neumann's minimax theorem, into a single choice of coefficients
\(\sigma_s\in[-1,1]\) such that the signed Kloosterman average has the
required sign uniformly for every \(t\in T\). Testing the positive Gram
matrices against the corresponding semidefinite multipliers forces a
contradiction unless \(|E|\) and \(|F|\) are \(O(q^m)\). The zero-radius sphere is handled separately by an incidence estimate, since over finite fields it may contain nonzero isotropic vectors.

\paragraph{Organization.}
Section~\ref{sec:Preliminaries} collects the Fourier-analytic preliminaries, evaluates the Fourier transform of quadratic shells, and constructs the frequency-shell Gram matrices. 

Section~\ref{sec3} establishes the quantitative semidefinite transfer
principle in Proposition~\ref{prop:quantitative-transfer}; its
constant-scale case, together with
Conjecture~\ref{conj:kloosterman-L1}, yields
Theorem~\ref{thm:kloosterman-L1-distance}.
Section~\ref{sec4} discusses Conjecture~\ref{conj:kloosterman-L1}: we
prove it in several model regimes, show its limitations at and near
full support, establish the universal \(q^{-\frac12}\)-level bound for
every fixed \(0<\alpha<1\), and prove the stronger
\(q^{-\frac13}\)-level bound for \(0<\alpha<\frac12\).  Through
Proposition~\ref{prop:quantitative-transfer}, these last two estimates
give parts \textup{(i)} and \textup{(ii)}, respectively, of
Corollary~\ref{cor:unconditional-baseline}.

\section{Preliminaries}\label{sec:Preliminaries}

Fix a nontrivial additive character $\chi$ of $\F_q$. We use the Fourier transform
\[
    \wh f(\xi)=\sum_{x\in V} f(x)\chi(-x\cdot \xi),
    \qquad \xi\in V.
\]
For $r\in\F_q$, write
\[
    S_r:=\{z\in V:Q(z)=r\}.
\]
Thus, $S_r$ is the $Q$-sphere of squared radius $r$, and $S_0$ is the zero radius sphere. Notice that, over a finite field, $S_0$ need not be the singleton $\{0\}$; it may contain nonzero isotropic vectors.

For $s\in\F_q^\times$, define the frequency sphere
\[
    \Sigma_s:=\{\xi\in V:Q(\xi)=4s\}.
\]
The factor $4$ appears because completing the square in the Fourier variable naturally produces $Q(\xi)/4$.

Let
\[
    \eps_m:=\eta(-1)^m\in\{\pm1\},
\]
where $\eta$ is the quadratic character of $\F_q^\times$. We use the classical, unnormalized Kloosterman sum
\[
    K(a):=\sum_{r\in\F_q^\times}\chi \bigg(r+\frac{a}{r} \bigg),
    \qquad a\in\F_q.
\]
In particular, $K(0)=-1$.

The following lemma bounds the number of pairs of points at distance
zero between a given pair of sets. A proof can be found in the literature, for example, as a consequence of a point-sphere incidence bound in \cite{CILRR2017, PPV}.
\begin{lemma}
\label{lem:zero-incidence}
There is a constant $C_m$ depending only on $m$ such that, for all $A,B\subset V$,
\[
    \big| \{(a,b)\in A\times B:Q(a-b)=0\}\big|
    \le \frac{|A||B|}{q}+C_mq^m(|A||B|)^{\frac{1}{2}}.
\]
\end{lemma}

The next lemma is the eigenvalue computation for the radial kernels that
underlie the Delsarte--linear programming method. For each frequency shell \(\Sigma_s\), the sum of the characters \(\chi(-z\cdot\xi)\) over
\(\xi\in\Sigma_s\) defines a translation-invariant kernel on \(V\).  Because this kernel depends only on the value of \(Q(z)\), it is constant on the quadratic distance classes \(S_r\), and its values are the eigenvalues that will appear in the linear constraints for the distance distributions.

In the semidefinite version used below, these shell eigenvalues are paired with the positive semidefinite Gram matrices obtained by restricting \(\widehat{\mathbf 1_E}\) and \(\widehat{\mathbf 1_F}\) to \(\Sigma_s\). A signed average over the shells will then be chosen by the minimax, or linear programming separation, step so that distances in the proposed support \(T\) receive a negative contribution. Thus, Lemma~\ref{lem:shell-eigenvalue} is the bridge between Fourier positivity on frequency shells and the Kloosterman coefficients \(K(su)\) that drive the later linear programming inequality.
\begin{lemma}
\label{lem:shell-eigenvalue}
For $s\in\F_q^\times$ and $z\in V$,
\[
    \sum_{\xi\in\Sigma_s}\chi(-z\cdot\xi)
    =q^{2m-1}\1_{z=0}+\eps_m q^{m-1}K(sQ(z)).
\]
\end{lemma}

\begin{proof}
By orthogonality,
\[
\sum_{\xi\in\Sigma_s}\chi(-z\cdot\xi)
=q^{-1}\sum_{\rho\in\F_q}\chi(-\rho s)
\sum_{\xi\in V}\chi\!\left(\frac{\rho}{4}Q(\xi)-z\cdot \xi\right).
\]
The term $\rho=0$ contributes $q^{2m-1}\1_{z=0}$. For $\rho\ne0$, completing the square in each coordinate gives
\[
    \sum_{\xi\in V}\chi\!\left(\frac{\rho}{4}Q(\xi)-z\cdot \xi\right)
    =\eps_m q^m \chi\bigg(- \frac{Q(z)}{\rho} \bigg).
\]
Therefore, the nonzero part is
\[
    \eps_m q^{m-1}
    \sum_{\rho\in\F_q^\times}\chi \bigg(-\rho s- \frac{Q(z)}{\rho} \bigg).
\]
After the change of variables $r=-\rho s$, this becomes
\[
    \eps_m q^{m-1}\sum_{r\in\F_q^\times}\chi \bigg(r+\frac{sQ(z)}{r} \bigg)
    =\eps_m q^{m-1}K(sQ(z)).
\]
\end{proof}
We now pass from the geometric distance information of the pair
\(E,F\) to frequency-side matrices supported on the shells \(\Sigma_s\).
The point is to encode the two self-correlations of \(E\) and \(F\), together with their cross-correlation, into a \(2\times2\) Gram matrix for each frequency radius \(s\in\F_q^\times\). Positivity of these Gram matrices will later be tested against a carefully chosen semidefinite multiplier. The second purpose of the subsection is to make the entries of these matrices explicit. Lemma~\ref{lem:shell-eigenvalue} converts each frequency-shell sum into a radial kernel depending only on \(Q(x-y)\), and in even dimension that radial kernel is the classical Kloosterman sum. Thus, the assumption \(\Delta_Q^\times(E,F)\subseteq T\) is transformed into a Kloosterman expansion supported on the set \(T\).

Assume $|E|=|F|=L$ and
\[
    \Delta_Q^\times(E,F)\subseteq T\subseteq\F_q^\times.
\]
Define normalized distance variables by
\begin{align*}
    a_u^E&:=\frac1L \big| \{(x,x')\in E^2:x\ne x',\ Q(x-x')=u\} \big| ,\qquad u\in\F_q,\\
    a_u^F&:=\frac1L \big| \{(y,y')\in F^2:y\ne y',\ Q(y-y')=u\} \big| ,\qquad u\in\F_q,\\
    h^{E,F}&:=\frac{|E\cap F|}{L},\\
    b_0^{E,F}&:=\frac1L \big| \{(x,y)\in E\times F:x\ne y,\ Q(x-y)=0\} \big|,\\
    b_t^{E,F}&:=\frac1L \big|\{(x,y)\in E\times F:Q(x-y)=t\} \big|,\qquad t\in T.
\end{align*}
These quantities satisfy \(h^{F,E}=h^{E,F}\) and
\(b_u^{F,E}=b_u^{E,F}\). In particular,
\(h^{E,E}=1\) and \(b_u^{E,E}=a_u^E\) for every \(u\in\F_q\).
Moreover, the support assumption gives
\[
    \sum_{t\in T}b_t^{E,F}=L-h^{E,F}-b_0^{E,F}.
\]
We next record the basic positivity statement. It is simply the observation that the three quantities \(A_s^E,A_s^F\), and \(A_s^{E,F}\) are the entries of the Gram matrix of the two vectors obtained by restricting \(\widehat{\mathbf 1_E}\) and \(\widehat{\mathbf 1_F}\) to the frequency sphere \(\Sigma_s\).

\begin{lemma}
\label{lem:frequency-gram-psd}
Let \(E,F\subset V\) with \(|E|=|F|=L>0\). For \(s\in\F_q^\times\), define
\[
    A_s^{E,F}:=\frac{1}{Lq^{m-1}}
    \sum_{\xi\in\Sigma_s}
    \wh{\1_E}(\xi)\ol{\wh{\1_F}(\xi)}, 
\]
and 
\[
A_s^E:=A_s^{E,E}=\frac{1}{Lq^{m-1}}
    \sum_{\xi\in\Sigma_s}|\wh{\1_E}(\xi)|^2, \qquad A_s^F:=A_s^{F,F}=\frac{1}{Lq^{m-1}}
    \sum_{\xi\in\Sigma_s}|\wh{\1_F}(\xi)|^2.
\]
Then, the Gram matrix
\[
    G_s(E,F):=
    \begin{pmatrix}
        A_s^E&A_s^{E,F}\\
        \ol{A_s^{E,F}}&A_s^F
    \end{pmatrix}
\]
is positive semidefinite. In particular,
\[
    A_s^E\ge0,\qquad A_s^F\ge0,\qquad
    |A_s^{E,F}|^2\le A_s^E A_s^F.
\]
\end{lemma}

\begin{proof}
For each \(\xi\in\Sigma_s\), put
\[
    v_\xi:=
    \begin{pmatrix}
        \wh{\1_E}(\xi)\\
        \wh{\1_F}(\xi)
    \end{pmatrix}.
\]
Then
\[
    v_\xi v_\xi^{\ast}
    =
    \begin{pmatrix}
        |\wh{\1_E}(\xi)|^2&
        \wh{\1_E}(\xi)\ol{\wh{\1_F}(\xi)}\\
        \wh{\1_F}(\xi)\ol{\wh{\1_E}(\xi)}&
        |\wh{\1_F}(\xi)|^2
    \end{pmatrix},
\]
and hence, by the definitions of \(A_s^E,A_s^F\), and \(A_s^{E,F}\),
\[
    G_s(E,F)
    =
    \frac{1}{Lq^{m-1}}
    \sum_{\xi\in\Sigma_s} v_\xi v_\xi^{\ast}.
\]
Each matrix \(v_\xi v_\xi^{\ast}\) is rank-one positive semidefinite.  Since \(Lq^{m-1}>0\), the matrix \(G_s(E,F)\) is a positive scalar multiple of a sum of positive semidefinite matrices, and is therefore positive semidefinite.

Equivalently, for every \(c=(c_1,c_2)\in\mathbb C^2\),
\[
    c^{\ast}G_s(E,F)c
    =
    \frac{1}{Lq^{m-1}}
    \sum_{\xi\in\Sigma_s}
    \left|
        \ol{c_1}\wh{\1_E}(\xi)
        +
        \ol{c_2}\wh{\1_F}(\xi)
    \right|^2
    \ge0.
\]
The inequalities \(A_s^E\ge0\), \(A_s^F\ge0\), and
\(|A_s^{E,F}|^2\le A_s^E A_s^F\) follow from positive semidefiniteness of a
\(2\times2\) Hermitian matrix.
\end{proof}
We now evaluate the Gram entries in physical space. The previous lemma established positivity; the next lemma supplies the arithmetic content by using Lemma~\ref{lem:shell-eigenvalue} to rewrite the entries as Kloosterman sums weighted by the normalized distance multiplicities.

\begin{lemma}
\label{lem:gram-kloosterman-expansion}
Assume that \(|E|=|F|=L\) and that
\[
    \Delta_Q^\times(E,F)\subseteq T\subseteq \F_q^\times .
\]
Then, for every \(s\in\F_q^\times\),
\begin{align}
    A_s^E
    &=q^m+\eps_m\left(
        -1-a_0^E+\sum_{u\in\F_q^\times}K(su)a_u^E
    \right),\label{eq:self-E}\\
    A_s^F
    &=q^m+\eps_m\left(
        -1-a_0^F+\sum_{u\in\F_q^\times}K(su)a_u^F
    \right),\label{eq:self-F}\\
    A_s^{E,F}
    &=h^{E,F} q^m+\eps_m\left(
        -h^{E,F}-b_0^{E,F}
        +\sum_{t\in T}K(st)b_t^{E,F}
    \right).\label{eq:cross}
\end{align}
In particular, \(A_s^{E,F}\) is real.
\end{lemma}

\begin{proof}
We first prove the identity for \(A_s^E\). By the definition of the Fourier transform,
\[
    \sum_{\xi\in \Sigma_s}|\wh{\1_E}(\xi)|^2
    =
    \sum_{x,x'\in E}
    \sum_{\xi\in \Sigma_s}
    \chi(-(x-x')\cdot \xi).
\]
Applying Lemma~\ref{lem:shell-eigenvalue} with \(z=x-x'\), we get
\[
    \sum_{\xi\in \Sigma_s}
    \chi(-(x-x')\cdot \xi)
    =
    q^{2m-1}\1_{x=x'}
    +\eps_m q^{m-1}K(sQ(x-x')).
\]
Since \(K(0)=-1\), the diagonal pairs \(x=x'\) contribute
\[
    Lq^{2m-1}+\eps_m q^{m-1}LK(0)=L(q^{2m-1}-\eps_m q^{m-1}).
\]
The off-diagonal pairs with \(Q(x-x')=0\) and the remaining
off-diagonal pairs, grouped according to
\(u=Q(x-x')\in\F_q^\times\), contribute, respectively,
\[
    -\eps_m q^{m-1}L a_0^E
    \qquad\text{and}\qquad
    \eps_m q^{m-1}L
    \sum_{u\in\F_q^\times}K(su)a_u^E.
\]
Therefore,
\[
    \sum_{\xi\in \Sigma_s}|\wh{\1_E}(\xi)|^2
    =
    Lq^{2m-1}
    +\eps_m q^{m-1}L
    \left(
        -1-a_0^E+\sum_{u\in\F_q^\times}K(su)a_u^E
    \right).
\]
Dividing by \(Lq^{m-1}\) gives \eqref{eq:self-E}. The proof of
\eqref{eq:self-F} is identical.

It remains to prove the cross identity. By definition,
\[
    \sum_{\xi\in \Sigma_s}
    \wh{\1_E}(\xi)\ol{\wh{\1_F}(\xi)}
    =
    \sum_{x\in E}\sum_{y\in F}
    \sum_{\xi\in \Sigma_s}
    \chi(-(x-y)\cdot \xi).
\]
Applying Lemma~\ref{lem:shell-eigenvalue} with \(z=x-y\), the pairs
with \(x=y\) contribute
\[
    |E\cap F|(q^{2m-1}-\eps_m q^{m-1})
    =h^{E,F}L(q^{2m-1}-\eps_m q^{m-1}).
\]
The non-diagonal pairs with \(Q(x-y)=0\) and the pairs with
\(Q(x-y)=t\in T\) contribute, respectively,
\[
    -\eps_m q^{m-1}L b_0^{E,F}
    \qquad\text{and}\qquad
    \eps_m q^{m-1}L\sum_{t\in T}K(st)b_t^{E,F}.
\]
Hence
\[
    \sum_{\xi\in \Sigma_s}
    \wh{\1_E}(\xi)\ol{\wh{\1_F}(\xi)}
    =
    h^{E,F}Lq^{2m-1}
    +\eps_m q^{m-1}L
    \left(-h^{E,F}-b_0^{E,F}
        +\sum_{t\in T}K(st)b_t^{E,F}
    \right).
\]
Dividing by \(Lq^{m-1}\) gives \eqref{eq:cross}.

Finally, since $\ol{K(a)}=K(a)$, the sum \(K(a)\) is
real for every \(a\in\F_q\). Since \(h^{E,F},b_0^{E,F},b_t^{E,F}\) are real, the formula \eqref{eq:cross} shows that \(A_s^{E,F}\) is real.
\end{proof}

\section{The semidefinite transfer principle}\label{sec3}
It suffices to prove the balanced formulation, since one may pass to equal-size subsets of $E$ and $F$ without increasing the set of nonzero distances.

First suppose $T=\varnothing$. Then every pair $(x,y)\in E\times F$ has $Q(x-y)=0$. Lemma~\ref{lem:zero-incidence}, applied to $(E,F)$, gives
\[
    L^2\le \frac{L^2}{q}+C_mq^mL,
\]
and therefore, $L\le C_m' q^m$ after increasing the constant. Hence, we may assume that $T\ne\varnothing$.

We first isolate the only consequence of Conjecture~\ref{conj:kloosterman-L1} that will be used in the semidefinite averaging argument. This is the linear programming separation step: it produces one fixed system of coefficients in the frequency variable $s$, independent of $t$, whose Kloosterman averages have the sign needed for the argument.

\begin{claim}
\label{claim:kloosterman-separation}
Assume Conjecture~\ref{conj:kloosterman-L1}. Let $\varnothing\ne T\subseteq\mathbb{F}_q^\times$ satisfy $|T|\le \alpha(q-1)$. Then there exist coefficients
\[
    \sigma_s\in[-1,1],
    \qquad s\in\mathbb{F}_q^\times,
\]
such that
\[
    \frac{\eps_m}{q-1}
    \sum_{s\in\mathbb{F}_q^\times}\sigma_sK(st)
    \le -c_\alpha
    \qquad(t\in T).
\]
\end{claim}

We postpone the proof of the claim and first show how it implies the balanced estimate. Choose coefficients $\sigma_s\in[-1,1]$ satisfying the conclusion of Claim~\ref{claim:kloosterman-separation}. We insert these coefficients into the positive semidefinite frequency-sphere Gram matrices.

For each $s\in\mathbb{F}_q^\times$, consider
\[
    H_s:=\frac1{q-1}
    \begin{pmatrix}
        1&\sigma_s\\
        \sigma_s&1
    \end{pmatrix}.
\]
Since $|\sigma_s|\le1$, the eigenvalues of $H_s$ are $\frac{1+\sigma_s}{q-1}$ and $\frac{1-\sigma_s}{q-1}$, both nonnegative. Thus, $H_s$ is positive semidefinite. By Lemma~\ref{lem:frequency-gram-psd}, $G_s(E,F)$ is also positive semidefinite. Hence
\[
    0\le \sum_{s\in\mathbb{F}_q^\times}\langle H_s,G_s(E,F)\rangle,
\]
where $\langle U,V\rangle=\Tr(U V)$. Since $A_s^{E,F}$ is real, this gives
\[
    0\le
    \frac1{q-1}
    \sum_{s\in\mathbb{F}_q^\times}
    \left(A_s^E+A_s^F+2\sigma_s A_s^{E,F}\right).
\]
We now substitute the Kloosterman expansions from Lemma~\ref{lem:gram-kloosterman-expansion} and estimate the resulting terms.

The nonzero self-distance terms are harmless. For every $v\in\mathbb{F}_q^\times$,
\[
    \sum_{s\in\mathbb{F}_q^\times}K(sv)=1.
\]
Indeed,
\[
    \sum_{s\in\mathbb{F}_q^\times}K(sv)
    =\sum_{r\in\mathbb{F}_q^\times}\chi(r)
      \sum_{s\in\mathbb{F}_q^\times}\chi(sv/r)
    =-\sum_{r\in\mathbb{F}_q^\times}\chi(r)=1.
\]
Therefore, the total contribution of the nonzero self-distance terms from $A_s^E$ and $A_s^F$ is bounded in absolute value by
\[
    \frac{2L}{q-1}.
\]
The nonzero cross-distance terms give the main negative contribution.  By Claim~\ref{claim:kloosterman-separation} and the nonnegativity of the multiplicities $b_t$,
\[
    \frac{2\eps_m}{q-1}
    \sum_{t\in T}b_t^{E,F}
    \sum_{s\in\mathbb{F}_q^\times}\sigma_sK(st)
    \le -2c_\alpha\sum_{t\in T}b_t^{E,F}.
\]
It remains to control the constant terms and the zero-distance terms.  Lemma~\ref{lem:zero-incidence}, applied to $(E,E)$, $(F,F)$, and $(E,F)$, gives
\[
    1+a_0^E\le \frac{L}{q}+C_mq^m,
    \qquad
    1+a_0^F\le \frac{L}{q}+C_mq^m, \qquad h^{E,F}+b_0^{E,F}
    \le \frac{L}{q}+C_m q^m.
\]
Also \(0\le h^{E,F}\le1\) and $|\sigma_s|\le1$. Hence all constant and zero-distance contributions are bounded in absolute value by
\[
    C_m\left(q^m+\frac Lq\right),
\]
after increasing $C_m$ if necessary.

Since $\Delta_Q^\times(E,F)\subseteq T$, the normalized cross-pair count decomposes as
\[
    L=h^{E,F}+b_0^{E,F}
    +\sum_{t\in T}b_t^{E,F}.
\]
Thus,
\[
    \sum_{t\in T} b_t^{E,F}
    =L-h^{E,F}-b_0^{E,F}
    \ge L-C_mq^m-\frac Lq.
\]
Combining the preceding estimates with the positive semidefinite inequality yields
\[
    0\le
    C_m\left(q^m+\frac Lq\right)
    +\frac{2L}{q-1}
    -2c_\alpha\left(L-C_mq^m-\frac Lq\right).
\]
For all sufficiently large \(q\), depending only on \(m\) and \(\alpha\), the terms involving \(\frac{L}{q}\) and \(\frac{L}{q-1}\) can be absorbed into the negative term, and we obtain
\[
    L\le C_{m,\alpha}q^m.
\]
After enlarging $C_{m,\alpha}$ to absorb the finitely many remaining odd prime powers, the balanced estimate follows.

It remains to prove Claim~\ref{claim:kloosterman-separation}.

\paragraph{Proof of Claim~\ref{claim:kloosterman-separation}}
Let
\[
    \mathcal C
    :=
    \left\{(\tau_s)_{s\in\mathbb{F}_q^\times}:
    -1\le \tau_s\le1\text{ for every }s\in\mathbb{F}_q^\times
    \right\}.
\]
Equivalently,
\[
    \mathcal C=
    \underbrace{[-1,1]\times\cdots\times[-1,1]}_{q-1\text{ times}}.
\]
For $\lambda\in\Prob(T)$ and $\tau\in\mathcal C$, define
\[
    B(\lambda,\tau)
    :=
    \frac1{q-1}
    \sum_{s\in\mathbb{F}_q^\times}
    \tau_s
    \sum_{t\in T}\lambda_tK(st).
\]

For fixed $\lambda\in\Prob(T)$, put
\[
    R_s(\lambda):=\sum_{t\in T}\lambda_tK(st).
\]
Since the Kloosterman sum $K$ is real-valued, each $R_s(\lambda)$ is real. Hence
\[
    \sup_{\tau\in\mathcal C}B(\lambda,\tau)
    =
    \frac1{q-1}
    \sum_{s\in\mathbb{F}_q^\times}|R_s(\lambda)|
    =
    \frac1{q-1}
    \sum_{s\in\mathbb{F}_q^\times}
    \left|\sum_{t\in T}\lambda_tK(st)\right|.
\]
Indeed, for every admissible $\tau$, one has
\[
    \tau_sR_s(\lambda)\le |R_s(\lambda)|
\]
term by term, while equality is obtained by choosing
$\tau_s=\sgn(R_s(\lambda))$, with any value in $[-1,1]$ when $R_s(\lambda)=0$.

By Conjecture~\ref{conj:kloosterman-L1}, the last quantity is at least $c_\alpha$, uniformly in $\lambda$.  Therefore,
\[
    \inf_{\lambda\in\Prob(T)}
    \sup_{\tau\in\mathcal C}
    B(\lambda,\tau)
    \ge c_\alpha.
\]
We now use the finite-dimensional von Neumann's minimax theorem, see \cite{Fan53, Nikaido54, Sion58}. If $X$ and $Y$ are compact convex subsets of finite-dimensional real vector spaces, and if $\Phi:X\times Y\to\mathbb R$ is continuous and affine in each variable separately, then
\[
    \inf_{x\in X}\sup_{y\in Y}\Phi(x,y)
    =
    \sup_{y\in Y}\inf_{x\in X}\Phi(x,y).
\]
We apply this with
\[
    X=\Prob(T),
    \qquad
    Y=\mathcal C,
    \qquad
    \Phi=B.
\]
Here, \(\Prob(T)\) denotes the simplex defined in \eqref{definition-simplex}, which is compact and convex.
 The set $\mathcal C$ is a product of $q-1$ closed intervals, hence is compact and convex. Finally, $B(\lambda,\tau)$ is bilinear, because it is a finite sum of terms of the form $\lambda_t\tau_sK(st)$ with $K(st)$ fixed. Thus, minimax gives
\[
    \inf_{\lambda\in\Prob(T)}
    \sup_{\tau\in\mathcal C}
    B(\lambda,\tau)
    =
    \sup_{\tau\in\mathcal C}
    \inf_{\lambda\in\Prob(T)}
    B(\lambda,\tau).
\]
Consequently,
\[
    \sup_{\tau\in\mathcal C}
    \inf_{\lambda\in\Prob(T)}
    B(\lambda,\tau)
    \ge c_\alpha.
\]
The supremum is attained, since $\mathcal C$ is compact and the function
\[
    \tau\mapsto \inf_{\lambda\in\Prob(T)}B(\lambda,\tau)
\]
is continuous; indeed, because $B$ is linear in $\lambda$, this infimum is the minimum of the finitely many continuous functions obtained by taking $\lambda=\delta_t$, $t\in T$. Hence there exists one choice $\tau=(\tau_s)_{s\in\mathbb{F}_q^\times}\in\mathcal C$ such that
\[
    B(\lambda,\tau)\ge c_\alpha
    \qquad\text{for every }\lambda\in\Prob(T).
\]
Taking $\lambda=\delta_t$, the point mass at $t\in T$, gives
\[
    \frac1{q-1}
    \sum_{s\in\mathbb{F}_q^\times}\tau_sK(st)
    \ge c_\alpha
    \qquad(t\in T).
\]
Finally, define
\[
    \sigma_s:=-\eps_m\tau_s,
    \qquad s\in\mathbb{F}_q^\times.
\]
Since $\eps_m=\pm1$, we still have $\sigma_s\in[-1,1]$. Moreover, for every $t\in T$,
\[
    \frac{\eps_m}{q-1}
    \sum_{s\in\mathbb{F}_q^\times}\sigma_sK(st)
    =
    -\frac1{q-1}
    \sum_{s\in\mathbb{F}_q^\times}\tau_sK(st)
    \le -c_\alpha.
\]
This proves the claim, and completes the proof of the theorem.

\begin{proof}[Proof of Proposition~\ref{prop:quantitative-transfer}]
Under the hypothesis of Proposition~\ref{prop:quantitative-transfer},
the same minimax argument gives coefficients
\(\sigma_s\in[-1,1]\) such that
\[
    \frac{\eps_m}{q-1}
    \sum_{s\in\mathbb F_q^\times}\sigma_sK(st)
    \le -c_{\alpha,\theta}q^{-\theta}
    \qquad(t\in T).
\]
Repeating the positive semidefinite averaging argument for balanced
sets of size \(L\) therefore gives
\[
    0\le
    C_m\left(q^m+\frac Lq\right)
    +\frac{2L}{q-1}
    -2c_{\alpha,\theta}q^{-\theta}
      \left(L-C_mq^m-\frac Lq\right).
\]
Since \(\theta<1\), the terms \(L/q\) and \(L/(q-1)\) can be absorbed
into the negative \(Lq^{-\theta}\) term once \(q\) is sufficiently
large in terms of \(m,\alpha\), and \(\theta\). It follows that
\[
    L\le C_{m,\alpha,\theta}q^{m+\theta}.
\]
Enlarging the constant handles the finitely many remaining odd prime
powers, and passing to equal-size subsets proves the stated
unbalanced formulation.
\end{proof}

 \begin{remark}
The restriction to \(d=2m\) is not merely notational. It enters at the
point where the Fourier transform of a quadratic shell is evaluated.  In even dimension the quadratic Gauss-sum factor contains \(\eta(\rho)^{2m}=1\), so the nonzero part of the shell transform is governed by the classical Kloosterman sum
\[
        K(a)=\sum_{r\in \F_q^\times}\chi\bigg(r+\frac{a}{r} \bigg).
\]
Thus, Conjecture~\ref{conj:kloosterman-L1} is exactly the analytic input
needed by the linear programming argument in this paper.

In odd dimension the same calculation does not give the same Kloosterman sum. The extra factor \(\eta(\rho)\) survives in the Gauss-sum computation, and the corresponding shell eigenvalues involve a Sali\'e-type sum rather than the classical Kloosterman sum.  Consequently Conjecture~\ref{conj:kloosterman-L1}
has no direct implication for the odd-dimensional distance problem.

This distinction is also forced by known examples. In odd dimensions there are arithmetic constructions of sets \(E\subset \F_q^d\) with
\[
        |E|\sim c q^{\frac{d+1}{2}}
        \qquad\text{and}\qquad
        |\Delta(E)|<c q
\]
for arbitrarily small fixed \(c>0\); see
\cite{HIKR11}. Since these sets are much larger than
\(q^{\frac{d}{2}}\), they rule out an odd-dimensional analogue of the present conclusion at the even-dimensional scale. This is consistent with the standard Fourier threshold \(q^{\frac{d+1}{2}}\) for obtaining all distances in the general odd-dimensional problem.
\end{remark}

\section{Discussions on Conjecture \ref{conj:kloosterman-L1}}\label{sec4}

Throughout this section, write
\[
    G:=\F_q^\times,
    \qquad N:=|G|=q-1,
\]
and, for a probability measure \(\lambda\) on a subset of \(G\), extend
\(\lambda\) by zero to all of \(G\) and put
\[
    R_\lambda(s):=\sum_{t\in G}\lambda_tK(st),
    \qquad
    \Phi_T(\lambda)=\frac1N\sum_{s\in G}|R_\lambda(s)|.
\]
We shall use the elementary autocorrelation identity
\begin{equation}
\label{eq:K-autocorrelation-discussion}
    \sum_{s\in G}K(st)K(su)
    =
    \begin{cases}
        q^2-q-1, & t=u,\\
        -q-1, & t\ne u,
    \end{cases}
\end{equation}
for \(t,u\in G\). Indeed, after the change of variables \(a=st\) and
\(v=u/t\), the left-hand side is
\[
    q\sum_{r\in G}\chi((1-v)r)-1,
\]
which gives the two cases above. Consequently, for every probability measure
\(\lambda\),
\begin{equation}
\label{eq:second-moment-discussion}
    \sum_{s\in G} R_\lambda(s)^2
    =q^2\sum_{t\in G}\lambda_t^2-q-1.
\end{equation}

\subsection{Examples supporting the conjecture}

The next two propositions give two simple unconditional regimes: concentrated measures and structured positive density supports. They are meant as tests of the conjecture, not as substitutes for the full uniform statement over all measures.

\begin{proposition}
\label{prop:discussion-concentrated-measures}
Fix \(A\ge1\). If \(\lambda\in\Prob(T)\) satisfies
\[
    |\operatorname{supp}\lambda|\le A q^{\frac{1}{2}},
\]
then, for all sufficiently large \(q\) depending only on \(A\),
\[
    \Phi_T(\lambda)\ge \frac1{4A}.
\]
In particular, point masses and probability measures supported on a bounded number of points satisfy the predicted lower bound.
\end{proposition}

\begin{proof}
By Cauchy's inequality,
\[
    \sum_{t\in G}\lambda_t^2
    \ge |\operatorname{supp}\lambda|^{-1}
    \ge A^{-1}q^{-\frac{1}{2}}.
\]
Using \eqref{eq:second-moment-discussion}, we get
\[
    \sum_{s\in G}R_\lambda(s)^2
    \ge A^{-1}q^{\frac{3}{2}}-q-1
    \ge \frac1{2A}q^{\frac{3}{2}}
\]
for all sufficiently large \(q\). The Weil bound for Kloosterman sums~\cite{Weil48} gives
\(|R_\lambda(s)|\le2q^{\frac{1}{2}}\). Therefore
\[
    \Phi_T(\lambda)
    =\frac1N\sum_{s\in G}|R_\lambda(s)|
    \ge
    \frac{\sum\limits_{s\in G} R_\lambda(s)^2}{2q^{\frac{1}{2}}N}
    \ge \frac1{4A}.
\]
\end{proof}

\begin{proposition}
\label{prop:discussion-coset-unions}
Let \(H\le G\) be a multiplicative subgroup of index \(d\), and let
\(U\subseteq G\) be a union of exactly \(r\) cosets of \(H\), where
\(1\le r<d\). Let \(\lambda=|U|^{-1}\1_U\). Then
\[
    \Phi_U(\lambda)
    \ge \sqrt{\frac{d-r}{rd}}.
\]
In particular, if \(\frac{r}{d}\le\alpha<\frac{1}{2}\) and \(r\le R\), then
\(|U|\le\alpha(q-1)\) and
\[
    \Phi_U(\lambda)\ge \sqrt{\frac{1-\alpha}{R}}.
\]
\end{proposition}

\begin{proof}
Since \(U\) is a union of \(H\)-cosets, we have \(hU=U\) for every \(h\in H\).
Thus \(R_\lambda(sh)=R_\lambda(s)\). Hence, \(R_\lambda\) is constant on the \(d\) cosets of \(H\); call the corresponding values \(c_1,\ldots,c_d\). Using the elementary inequality \(\|c\|_1\ge \|c\|_2\), we obtain
\[
    \Phi_U(\lambda)=\frac1d\sum_{j=1}^d |c_j|
    \ge \frac1d\left(\sum_{j=1}^dc_j^2\right)^{\frac{1}{2}}
    =\left(\frac1d\cdot\frac1N\sum_{s\in G}R_\lambda(s)^2\right)^{\frac{1}{2}}.
\]
Since \(|U|= \frac{r N}{d}\), we have \(\sum\limits_{t\in G}\lambda_t^2= \frac{1}{|U|}= \frac{d}{r N}\). By
\eqref{eq:second-moment-discussion},
\[
    \frac1N\sum_{s\in G}R_\lambda(s)^2
    =\frac1N \bigg( q^2 \frac{d}{rN}-q-1 \bigg)
    =\frac{1}{N^2} \bigg( \bigg( \frac{d}{r}-1 \bigg)q^2+1 \bigg).
\]
Therefore
\[
    \Phi_U(\lambda)
    \ge
    \frac{1}{\sqrt d\,N} \sqrt{\bigg( \frac{d}{r}-1 \bigg)q^2+1}
    \ge \sqrt{\frac{d-r}{rd}}.
\]
The final assertion follows from \(|U|=\frac{r}{d} N \le\alpha N\) and 
\[
\frac{d-r}{rd}=\frac{1}{r} \bigg( 1-\frac{r}{d} \bigg) \geq \frac{1-\alpha}{R}. \]
\end{proof}

\subsection{Endpoint and near full support obstructions}

The next examples explain why the density restriction in the conjecture cannot be pushed to the endpoint. More precisely, they show failure at \(\alpha=1\) and failure for moving support densities tending to \(1\).  They do not, by themselves, disprove any fixed-density statement with \(\alpha<1\).

\begin{proposition}
\label{prop:discussion-full-support}
Let \(T=G\), and let \(\lambda\) be the uniform probability measure on \(G\).
Then
\[
    R_\lambda(s)=\frac1N\qquad(s\in G),
    \qquad
    \Phi_G(\lambda)=\frac1N.
\]
In particular, no positive lower bound independent of \(q\) can hold when \(|T|=q-1\) is allowed.
\end{proposition}

\begin{proof}
For fixed \(s\in G\), multiplication by \(s\) permutes \(G\), and hence
\[
    R_\lambda(s)=\frac1N\sum_{a\in G}K(a).
\]
Expanding the Kloosterman sum gives
\[
    \sum_{a\in G}K(a)
    =\sum_{r\in G}\chi(r)\sum_{a\in G}\chi\bigg(\frac{a}{r} \bigg)
    =(-1)(-1)=1.
\]
Thus, \(R_\lambda(s)=\frac{1}{N}\) for all \(s\), and the formula for \(\Phi_G(\lambda)\)
follows.
\end{proof}

\begin{proposition}
\label{prop:discussion-near-full}
Let \(q\to\infty\) through odd prime powers, put \(N=q-1\), and let
\(r=r(q)\) satisfy
\[
    0\le r\le \frac N2,
    \qquad
    \frac rN\to0.
\]
If \(T\subseteq G\) has \(|T|=N-r\) and \(\lambda=|T|^{-1}\1_T\), then
\[
    \frac{|T|}{q-1}=1-\frac rN\to1
\]
and
\[
    \Phi_T(\lambda)
    \le
    \left(
        \frac{q^2}{N(N-r)}-\frac{q^2-1}{N^2}
    \right)^{\frac{1}{2}}
    =o(1).
\]
\end{proposition}

\begin{proof}
By Cauchy's inequality and \eqref{eq:second-moment-discussion},
\[
\begin{aligned}
    \Phi_T(\lambda)
    &\le \left(\frac1N\sum_{s\in G}R_\lambda(s)^2\right)^{\frac{1}{2}}  =\left(
        \frac{q^2}{N(N-r)}-\frac{q^2-1}{N^2}
      \right)^{\frac{1}{2}},
\end{aligned}
\]
because \(\sum\limits_t\lambda_t^2=\frac{1}{N-r}\).  Since \(q=N+1\), the expression inside the square root is
\[
    \frac{(N+1)^2}{N(N-r)}-\frac{N^2+2N}{N^2},
\]
which tends to \(0\) whenever \(\frac{r}{N}\to0\).
\end{proof}

\paragraph{The role of the density \(\boldsymbol{\frac12}\).}
The cutoff \(\frac12\) has different logical status in the conjecture
and in the cubic argument. In
Proposition~\ref{prop:discussion-theta-third}, it enters concretely
through
\[
    d_\alpha=\frac{(1-\alpha)(1-2\alpha)}{\alpha^2},
\]
which is positive exactly when \(\alpha<\frac12\); at
\(\alpha=\frac12\), the one-sided cubic-moment inequality used below
has no positive main term. Thus \(\alpha<\frac12\) is essential for
the present cubic proof. By contrast, the examples in this subsection
show failure only at full support or along densities tending to \(1\);
they do not show that a constant-scale lower bound fails for a fixed
\(\alpha\in(\frac12,1)\). The transfer principle,
Proposition~\ref{prop:quantitative-transfer}, itself works for every
fixed \(\alpha<1\). We formulate
Conjecture~\ref{conj:kloosterman-L1} in the range
\(\alpha<\frac12\) relevant to the main conditional theorem and make
no claim here, either positive or negative, about its fixed-density
extension beyond that range.

\subsection{A weaker lower bound holds at \texorpdfstring{$\theta=\frac{1}{2}$}{theta = 1/2}}
\label{subsec:discussion-theta-half}

Although Conjecture~\ref{conj:kloosterman-L1} asks for a constant lower bound when \(|T|\le\alpha(q-1)\) with \(\alpha<\frac{1}{2}\), a weaker estimate is immediate from the second moment.

\begin{proposition}
\label{prop:discussion-theta-half}
Fix \(0<\alpha<1\). Then, for every odd prime power \(q\), every nonempty
\(T\subseteq G\) with \(|T|\le\alpha(q-1)\), and every
\(\lambda\in\Prob(T)\),
\[
    \frac1{q-1}\sum_{s\in\F_q^\times}
    \left|\sum_{t\in T}\lambda_tK(st)\right|
    \ge
    \frac{1-\alpha}{2\alpha}\,q^{-\frac{1}{2}}.
\]
Thus, the weakened form of the conjecture holds with \(\theta=\frac{1}{2}\).
\end{proposition}

\begin{proof}
By Cauchy's inequality,
\[
    \sum_{t\in G}\lambda_t^2\ge \frac1{|T|}\ge \frac1{\alpha N}.
\]
Using \eqref{eq:second-moment-discussion}, we obtain
\[
    \sum_{s\in G}R_\lambda(s)^2
    \ge \frac{q^2}{\alpha N}-q-1
    =\frac{(1-\alpha)q^2+\alpha}{\alpha N}
    \ge \frac{1-\alpha}{\alpha}q.
\]
On the other hand, the Weil bound for Kloosterman sums~\cite{Weil48} gives
\(|K(a)|\le2q^{\frac{1}{2}}\) for \(a\ne0\).  Since \(s,t\in G\), this implies
\(|R_\lambda(s)|\le2q^{\frac{1}{2}}\) for all \(s\in G\). Therefore
\[
\begin{aligned}
    \Phi_T(\lambda)
    &=\frac1N\sum_{s\in G}|R_\lambda(s)|  \\
    &\ge \frac{\sum\limits_{s\in G}R_\lambda(s)^2}
            {N\max\limits_{s\in G}|R_\lambda(s)|}  \\
    &\ge \frac{1-\alpha}{2\alpha}\frac{q}{Nq^{\frac{1}{2}}}
     \ge \frac{1-\alpha}{2\alpha}q^{-\frac{1}{2}}.
\end{aligned}
\]
This proves the claim.
\end{proof}

Combining Proposition~\ref{prop:discussion-theta-half} with
Proposition~\ref{prop:quantitative-transfer} at
\(\theta=\frac12\) gives
Corollary~\ref{cor:unconditional-baseline}\textup{(i)} for the full
range \(0<\alpha<1\).

\subsection{A weaker lower bound holds at \texorpdfstring{$\theta=\frac{1}{3}$}{theta = 1/3}}
\label{subsec:discussion-theta-third}

The second-moment argument above applies to every fixed support density
strictly less than \(1\). In the range \(0<\alpha<\frac12\), a cubic
argument gives the stronger lower bound at scale \(q^{-\frac13}\).  We
first record the three auxiliary facts that will be used.

Let
\[
    \mathcal K:=\bigl(K(st)\bigr)_{s,t\in G},
\]
let \(\1\) denote the all-one column vector indexed by \(G\), and put
\(J:=\1 \, \1^{\mathsf T}\).

\begin{lemma}[Lemma~2.2, \cite{ZhangLP}]
\label{lem:discussion-K-square}
One has
\begin{equation*}
    \mathcal K\1=\1,
    \qquad
    \mathcal K^2=q^2I-(q+1)J.
\end{equation*}
\end{lemma}

\begin{lemma}[Lemma~2.3, \cite{ZhangLP}]
\label{lem:discussion-cubic-correlation}
For every \(t_1,t_2,t_3\in G\),
\begin{equation*}
    \left|
        \sum_{s\in G}K(st_1)K(st_2)K(st_3)
    \right|
    \le (q+1)^2.
\end{equation*}
\end{lemma}

\begin{lemma}[One-sided form of Lemma~2.5, \cite{ZhangLP}]
\label{lem:discussion-centered-cubic}
Let \(0<\alpha<\frac12\), let \(\Omega\) be a finite set, and let
\(X:\Omega\to\mathbb R\) satisfy
\[
    \frac1{|\Omega|}\sum_{\omega\in\Omega}X(\omega)=0.
\]
Suppose that \(A\ge0\), that \(X(\omega)\ge-A\) for every
\(\omega\in\Omega\), and that \(X(\omega)=-A\) on at least
\((1-\alpha)|\Omega|\) points. Then
\begin{equation*}
    \frac1{|\Omega|}\sum_{\omega\in\Omega}X(\omega)^3
    \ge
    A^3\frac{(1-\alpha)(1-2\alpha)}{\alpha^2}.
\end{equation*}
\end{lemma}

\begin{proposition}
\label{prop:discussion-theta-third}
Fix \(0<\alpha<\frac12\), and set
\[
    d_\alpha:=\frac{(1-\alpha)(1-2\alpha)}{\alpha^2}>0.
\]
Then, for every odd prime power \(q\), every nonempty \(T\subseteq G\)
with \(|T|\le\alpha(q-1)\), and every \(\lambda\in\Prob(T)\), one has
\begin{equation}
\label{eq:discussion-theta-third-exact}
    \Phi_T(\lambda)
    \ge
    d_\alpha^{\frac13}
    \left(\frac{q+1}{(q-1)^2}\right)^{\frac13}
    \ge
    d_\alpha^{\frac13}q^{-\frac13}.
\end{equation}
Thus, the weakened form of the conjecture holds with
\(\theta=\frac13\).
\end{proposition}

\begin{proof}
Extend \(\lambda\) by zero to \(G\), regard it as a column vector, and
write
\[
    R:=\mathcal K\lambda,
    \qquad
    S:=\norm{R}_1=\sum_{s\in G}|R(s)|.
\]
Thus \(R(s)=R_\lambda(s)\) and \(\Phi_T(\lambda)=\frac{S}{N}\).  Since
\(J\lambda=\1\), Lemma~\ref{lem:discussion-K-square} gives
\begin{equation}
\label{eq:discussion-KR-exact}
    \mathcal K R
    =\mathcal K^2\lambda
    =q^2\lambda-(q+1)\1.
\end{equation}
Define
\[
    a:=\frac{q^2}{q+1},
    \qquad
    u_t:=a\lambda_t,
    \qquad
    \rho:=\frac{a}{N}=\frac{q^2}{q^2-1}>1.
\]
Then
\begin{equation}
\label{eq:discussion-KR-u}
    (\mathcal K R)(t)=(q+1)(u_t-1),
    \qquad
    \frac1N\sum_{t\in G}u_t=\rho.
\end{equation}
Put \(X_t:=u_t-\rho\). Then
\[
    \frac1N\sum_{t\in G}X_t=0,
    \qquad
    X_t\ge-\rho\quad(t\in G).
\]
Moreover, \(X_t=-\rho\) whenever
\(t\notin\operatorname{supp}\lambda\). It follows from
\[
    |G\setminus\operatorname{supp}\lambda|
    \ge N-|T|
    \ge(1-\alpha)N,
\]
and Lemma~\ref{lem:discussion-centered-cubic} with
\(\Omega=G\) and \(A=\rho\) that
\begin{equation*}
    \sum_{t\in G}X_t^3
    \ge N\rho^3d_\alpha
    \ge Nd_\alpha.
\end{equation*}
Because \(\sum\limits_{t\in G}X_t=0\) and \(\rho-1>0\), we have
\begin{align*}
    \sum_{t\in G}(u_t-1)^3
    &=\sum_{t\in G}\bigl(X_t+(\rho-1)\bigr)^3 \notag\\
    &=\sum_{t\in G}X_t^3
      +3(\rho-1)\sum_{t\in G}X_t^2
      +N(\rho-1)^3 \notag\\
    &\ge Nd_\alpha.
\end{align*}
It follows from \eqref{eq:discussion-KR-u} that
\begin{equation}
\label{eq:discussion-KR-cubic-lower}
    \sum_{t\in G}\bigl((\mathcal K R)(t)\bigr)^3
    \ge(q+1)^3Nd_\alpha.
\end{equation}
On the other hand, expanding the cubic and applying
Lemma~\ref{lem:discussion-cubic-correlation}, we obtain
\begin{align}
\nonumber
    \left|\sum_{t\in G}\bigl((\mathcal K R)(t)\bigr)^3\right|
    &\le
      \sum_{s_1,s_2,s_3\in G}
      |R(s_1)R(s_2)R(s_3)| \bigg( 
      \bigg|\sum_{t\in G}
      K(ts_1)K(ts_2)K(ts_3)\bigg| \bigg) \\ \nonumber
    &\le(q+1)^2
      \sum_{s_1,s_2,s_3\in G}
      |R(s_1)R(s_2)R(s_3)| \\
    &=(q+1)^2S^3.
\label{eq:discussion-KR-cubic-upper}
\end{align}
Combining \eqref{eq:discussion-KR-cubic-lower} and
\eqref{eq:discussion-KR-cubic-upper}, we conclude that
\[
    S^3\ge(q+1)Nd_\alpha=(q^2-1)d_\alpha.
\]
Consequently,
\[
    \Phi_T(\lambda)=\frac SN
    \ge
    d_\alpha^{\frac13}
    \frac{(q^2-1)^{\frac13}}{q-1}
    =d_\alpha^{\frac13}
     \left(\frac{q+1}{(q-1)^2}\right)^{\frac13} \geq d_\alpha^{\frac13}
     q^{-\frac13}.
\]
This proves
\eqref{eq:discussion-theta-third-exact}.
\end{proof}

The positivity of \(d_\alpha\) is exactly where the restriction
\(\alpha<\frac12\) enters this proof. No
\(q^{-\frac13}\)-level conclusion is asserted here for
\(\frac12\le\alpha<1\). Combining the proposition with
Proposition~\ref{prop:quantitative-transfer} gives
Corollary~\ref{cor:unconditional-baseline}\textup{(ii)}.

\begin{proof}[Proof of Corollary~\ref{cor:unconditional-baseline}]
We obtain parts \textup{(i)} and \textup{(ii)} by applying Proposition~\ref{prop:quantitative-transfer} with $\theta=\frac12$ and $c_{\alpha, \frac12}=\frac{1-\alpha}{2\alpha}$ via Proposition~\ref{prop:discussion-theta-half}, and with $\theta=\frac13$ and $c_{\alpha, \frac13}=d_\alpha^{\frac13}$ via Proposition~\ref{prop:discussion-theta-third}, respectively
\end{proof}

\end{document}